\documentclass[10pt, a4paper,reqno]{amsart}

\usepackage[utf8]{inputenc}
\usepackage[T1]{fontenc}
\usepackage{paralist,charter}
\usepackage{amssymb}
\usepackage{amsmath}
\usepackage{amsthm}
\usepackage[ngerman,english]{babel}
\usepackage[noadjust]{cite}
\usepackage[hang]{caption}
\usepackage{fancyhdr}
\usepackage{chngcntr}
\usepackage{tikz}
\usepackage{mathrsfs}
\usepackage{mathabx}

\usepackage{todonotes}

\numberwithin{equation}{section}
\numberwithin{figure}{section}

\newtheoremstyle{mythm}{3pt}{3pt}{\itshape}{0pt}{\bfseries}{.}{0.5eM}{}
\theoremstyle{mythm}

\newtheorem{definition}{Definition}[section]
\newtheorem{thm}[definition]{Theorem}

\newtheorem{lem}[definition]{Lemma}
\newtheorem{cor}[definition]{Corollary}  
\newtheorem{prop}[definition]{Proposition}

\newtheoremstyle{myrem}{3pt}{3pt}{\normalfont}{0pt}{\bfseries}{.}{0.5em}{}
\theoremstyle{myrem}

\newtheorem{rem}[definition]{Remark}    
\newtheorem{example}[definition]{Example}

\newcounter{paranum}[section]

\newcommand{\mc}{\mathcal}

\renewcommand{\:}{\colon}

\newcommand{\Fol}{\mathcal{F}}

\renewcommand{\limsup}{\varlimsup}
\renewcommand{\liminf}{\varliminf}
\newcommand{\I}{\mathbb I}
\newcommand{\ZF}{[\Z]^{<\omega}}
\newcommand{\bb}{\mathbf b}
\newcommand{\dd}{\mathbf d}
\newcommand{\vv}{\mathbf v}
\newcommand{\uu}{\mathbf u}
\newcommand{\ww}{\mathbf w}

\newcommand{\ssq}{\ensuremath{\subseteq}}

\newcommand{\eps}{\ensuremath{\varepsilon}}

\newcommand{\Id}{\ensuremath{\mathrm{Id}}}

\newcommand{\supp}{\ensuremath{\mathrm{supp}}}

\newcommand{\alphlist}{\begin{list}{(\alph{enumi})}{\usecounter{enumi}\setlength{\parsep}{2pt}
      \setlength{\itemsep}{1pt} \setlength{\topsep}{5pt}
      \setlength{\partopsep}{3pt}}}
\newcommand{\arablist}{\begin{list}{(\arabic{enumi})}{\usecounter{enumi}\setlength{\parsep}{2pt}
          \setlength{\itemsep}{1pt} \setlength{\topsep}{5pt}
          \setlength{\partopsep}{3pt}}}
\newcommand{\romanlist}{\begin{list}{(\roman{enumi})}{\usecounter{enumi}\setlength{\parsep}{2pt}
              \setlength{\itemsep}{1pt} \setlength{\topsep}{5pt}
              \setlength{\partopsep}{3pt}}}
\newcommand{\Romanlist}{\begin{list}{(\Roman{enumi})}{\usecounter{enumi}\setlength{\parsep}{2pt}
              \setlength{\itemsep}{1pt} \setlength{\topsep}{5pt}
              \setlength{\partopsep}{3pt}}}
\newcommand{\bulletlist}{\begin{list}{$\bullet$}{\setlength{\parsep}{2pt}
                \setlength{\itemsep}{1pt} \setlength{\topsep}{5pt}
                \setlength{\partopsep}{3pt}\setlength{\leftmargin}{15pt}}} 
\newcommand{\Alphlist}{\begin{list}{(\Alph{enumi})}{\usecounter{enumi}\setlength{\parsep}{2pt}
      \setlength{\itemsep}{1pt} \setlength{\topsep}{5pt}
      \setlength{\partopsep}{3pt}}}
 \newcommand{\listend}{\end{list}}

\newcommand{\N}{\ensuremath{\mathbb{N}}} 
\newcommand{\R}{\ensuremath{\mathbb{R}}}
\newcommand{\Z}{\ensuremath{\mathbb{Z}}}

\title{On the continuity of Følner averages}

\author[Gabriel Fuhrmann]{Gabriel Fuhrmann}
\author[Maik Gr\"{o}ger]{Maik Gröger}
\author[Till Hauser]{Till Hauser}

\address[Gabriel Fuhrmann]{Department of Mathematical Sciences, Durham University, UK}
\email{gabriel.fuhrmann@durham.ac.uk}

\address[Maik Gr\"{o}ger]{Faculty of Mathematics and Computer Science, Jagiellonian University in Krakow, Poland
and FB03 - Mathematik und Informatik, University of Bremen, Germany}
\email{maik.groeger@im.uj.edu.pl}

\address[Till Hauser]{Facultad de Matemáticas, Pontificia Universidad Católica de Chile, Chile and Max-Planck-Institut für Mathematik - Bonn, Germany}
\email{hauser.math@mail.de}

\thanks{The research leading to these results has received funding from the Norwegian
Financial Mechanism 2014-2021 via the POLS grant no. 2020/37/K/ST1/02770.
This article was funded by the Deutsche Forschungsgemeinschaft (DFG, German Research Foundation) - 530703788.
MG is very grateful to D.~Kwietniak and H.~Pourmand for
several very helpful discussions concerning the notion of weak mean equicontinuity.
TH is thankful to M.~Schneider for the very valuable discussions
regarding thin F\o lner sequences and the appendix of this article.
The authors also thank the anonymous referee for valuable feedback and helpful suggestions.}
\begin{document}

\begin{abstract}
  It is known that if each point $x$ of a dynamical system is generic for some invariant measure $\mu_x$, then there is a strong connection between certain ergodic and topological properties of that system.
  In particular, if the acting group is abelian and the map $x\mapsto \mu_x$ is continuous, then every orbit closure is uniquely ergodic.
  
  In this note, we show that if the acting group is not abelian, 
  orbit closures may well support more than one ergodic measure even if $x\mapsto \mu_x$ is continuous.
  We provide examples of such a situation via actions of the
  group of all orientation-preserving homeomorphisms on the unit interval as well
  as the Lamplighter group.
  To discuss these examples, we need to extend the
  existing theory of weakly mean equicontinuous
  group actions to allow for multiple ergodic measures on orbit closures and
  to allow for actions of general amenable groups.
  These extensions are achieved by adopting an operator-theoretic approach.
\end{abstract}

\maketitle

\section{Introduction}

Temporal averages and their relation to invariant measures of a dynamical system
are fundamental in ergodic theory.
A classical notion intimately linked to this relationship is that of
generic points \cite{KryloffBogoliouboff1937}.
Given a dynamical system $(X,f)$ with $X$ a compact metric space
and  $f: X\to X$ a homeomorphism, a point $x\in X$ is called \emph{generic} for a 
measure $\mu$ (which is invariant under $f$) if for every continuous
function $\varphi: X\to\R$
\[
  \lim\limits_{n\to\infty}\frac{1}{n}\sum_{k=0}^{n-1} \varphi(f^k(x))=\int\limits_X\! \varphi \,d\mu.
\]

It is well known that for each ergodic measure $\mu$ there is
a set of full $\mu$-measure all of whose elements are 
generic for $\mu$. 
One may ask what happens if all points of
a dynamical system are generic for some invariant measure.
To the authors' knowledge, the first to investigate this question were Dowker and Lederer \cite{DowkerLederer1964}.
As they point out, already the classical theory of Kryloff and Bogoliouboff gives that genericity of every point together with minimality of the system implies unique ergodicity, see also \cite{Oxtoby1952}.
Indeed, weakening the assumption of minimality, they prove the following intriguing 
statement.
	
\begin{thm}[\cite{DowkerLederer1964}]
  Assume that $(X,f)$ has a unique minimal set and that all points are generic
  for some invariant measure.
  Then $(X,f)$ is either uniquely ergodic or there exist infinitely many
  ergodic measures.
\end{thm}

In 1981, Katznelson and Weiss \cite{KatznelsonWeiss1981} improved on this
result by showing that in the second case, there must exist uncountably many
ergodic measures.
Against the backdrop of this interesting interplay between ergodic and topological properties, it is natural to look for other assumptions which
have similar structural consequences when all points 
are generic. 

A most natural such assumption is continuity of the map $x\mapsto \mu_x$ which sends
each point $x\in X$ to the invariant measure $\mu_x$ it is generic for.
A priori, we can expect one of the following (non-exclusive) cases:
(a) no further rigidity (some but not all $\mu_x$ are ergodic), (b) all $\mu_x$ are ergodic
or (c) each orbit closure is uniquely ergodic.
Recent work of Downarowicz and Weiss \cite{DownarowiczWeiss2020},
Cai, Kwietniak, Li and Pourmand \cite{CaiKwietniakLiPourmand2022} as well 
as Xu and Zheng \cite{XuZheng2022}, shows the following.

\begin{thm}[\cite{DownarowiczWeiss2020,CaiKwietniakLiPourmand2022,XuZheng2022}]\label{thm:pointwise_uniq_erg}
  Suppose all points of $(X,f)$ are generic and that the map $x\mapsto\mu_x$
  is continuous.
  Then each orbit closure of $(X,f)$ is uniquely ergodic.
\end{thm}

In fact, this strong rigidity observed for $\Z$-actions extends to actions by
(countable) abelian groups, see \cite{CaiKwietniakLiPourmand2022} and \cite{XuZheng2022}
(note that strictly speaking, \cite{CaiKwietniakLiPourmand2022} only considers
integer actions but---as the authors point out---their techniques straightforwardly
extend to the general abelian setting). 
The main contribution of the present work is to show that this is no longer the case
when we consider actions by non-abelian groups.
Indeed, we provide concrete counter-examples involving the group of all orientation
preserving homeomorphisms on the unit interval as well as the Lamplighter group
showing that from the above mentioned cases also (a) and (b) (without pointwise unique ergodicity) can occur.
 
Before formulating our main result, we introduce some terminology. 
To be able to average along orbits (and to be able to make sense of the notion of genericity),
we need the notion of a (left) Følner sequence $(F_n)_{n\in\N}$ in the acting group $G$.
The precise nature of $(F_n)$ depends on the properties of $G$, see Section~\ref{sec:Folner_seq}.
When $G$ is a countable discrete amenable group (such as the Lamplighter group), the sets $F_n$ are  finite sets and a point $x\in X$ is $\mu$-\emph{generic} along $(F_n)_{n\in\N}$
(with $\mu$ an invariant measure) if
\[
    \lim\limits_{n\to\infty}1/|F_n|\cdot\sum_{g\in F_n} \varphi(gx)=\int\limits_X\!\varphi \,d\mu
\]
for every continuous function $\varphi: X\to\R$. 
Clearly, this is equivalent to
\[
  \lim_{n\to\infty}1/|F_n|\cdot \sum_{g\in F_n} \delta_{gx}=\mu,
\]
where the limit is taken in the weak*-topology.

With this, we can state our main result, see Section~\ref{sec: example} for the details.
 
\begin{thm}\label{thm: example}
  There exists an effective transitive action of the Lamplighter group $G$ on a compact metric 
  space $X$ with two trivial ergodic measures $\delta_{\hat\infty}$ and $\delta_{\check\infty}$
  where for each of the following (mutually exclusive) alternatives, there 
  is a left F\o lner sequence $(F_n)_{n\in\N}$ in $G$ such that
  $x\mapsto \lim_{n\to\infty}1/|F_n|\cdot \sum_{g\in F_n} \delta_{gx}=\mu_x$ is
  well-defined and
  \begin{itemize}
    \item $x\mapsto \mu_x$ is not continuous or
    \item $x\mapsto \mu_x$ is continuous and $\mu_x$ is not ergodic for any $x\in X\setminus \{\hat \infty,\check\infty\}$ or
    \item $x\mapsto \mu_x$ is continuous and $\mu_x$ is ergodic for some but not all $x\in X\setminus \{\hat \infty,\check\infty\}$ or
    \item $x\mapsto \mu_x$ is continuous and $\mu_x$ is ergodic for every $x\in X$.
 \end{itemize}
  Moreover, if $(F_n)_{n\in\N}$ is a right F\o lner sequence, then F\o lner
  averages along $(F_n)_{n\in\N}$ are not continuous.
  Specifically,
  \[
    \lim_{n\to \infty} 1/|F_n|\cdot \sum_{g\in F_n} \delta_{gx}
    =1/2\cdot \delta_{\hat \infty}
    +1/2\cdot \delta_{\check \infty} \qquad (x\in X\setminus \{\hat \infty,\check \infty\}).
  \]
\end{thm}

En passant, the above answers a question by Li, Ye and Yu \cite[Question 7]{LiYeYu2021}
in the negative: for general actions by locally compact $\sigma$-compact amenable
groups, the notions of Besicovitch-mean equicontinuity and Weyl-mean equicontinuity
do not coincide (for a background,
see \cite{LiYeYu2021, FuhrmannGroegerLenz2022} and references therein).
Further, our example also shows that the assumptions in
\cite[Theorem 1.7]{XuZheng2022} are necessary.

The last part of Theorem \ref{thm: example} has to be seen against the background of a general structural result, see Theorem~\ref{the:Weyl-weak-mean}:
if each $x\in X$ is $\mu_x$-generic along a
fixed \emph{right} F\o lner sequence and the map $x\mapsto\mu_x$
is continuous, then the conclusions of
Theorem~\ref{thm:pointwise_uniq_erg} still hold. 
In more general terms, this shows that already for the very basic ergodic properties 
considered in this article, there are fundamental differences between left and right F\o lner sequences.

Let us point out that we provide a second (and simpler) example
which shows that Theorem~\ref{thm:pointwise_uniq_erg} cannot
hold in general (that is, beyond actions of abelian groups), see Example~\ref{exmp: F-weak on interval}.
That example, however, does not allow for the case where
all $\mu_x$ are ergodic (case (b)) and furthermore, involves the less standard concept of thin F\o lner
sequences (introduced in \cite{SchneiderThom2018}), see Section~\ref{sec: thin Folner sequence} and the appendix for a background.

Finally, letting $\mc F$ be a fixed Følner sequence, the property that every 
point $x\in X$ is $\mu_x$-generic along $\mc F$ with $x\mapsto\mu_x$ continuous is intimately linked to what is referred to as
$\mc F$-weak mean equicontinuity, see Section~\ref{sec:weak_mean_equi}.
This notion was introduced by Zheng and Zheng for $\Z$-actions, see \cite{ZhengZheng2020}
and further \cite{CaiKwietniakLiPourmand2022,XuZheng2022}.
We generalize (parts of) \cite[Theorem 4.4]{CaiKwietniakLiPourmand2022}
as well as \cite[Theorems 3.5 and 4.3]{XuZheng2022} beyond the setting of countable 
discrete abelian/amenable groups, see Theorems \ref{thm: characterisations of weak mean equicontinuity} 
and \ref{the:Weyl-weak-mean}.
This allows us to discuss the novel examples from above in a unified framework.
We achieve this unification by associating to a general $\mc F$-weakly mean
equicontinuous action a natural bounded linear operator $S$ on $C(X)$ (the space of all continuous functions on $X$). 
This operator is a positive contractive projection 
(meaning that $S\varphi\geq 0$ if $\varphi\geq 0$, $S\boldsymbol{1}\leq 1$ and $S^2=S$)
and we obtain
\begin{thm}
The cases (a)--(c) from above translate to
\begin{enumerate}[(a)]
  \item $S(\varphi S\psi)=S(S\varphi S\psi)$ for all $\varphi,\psi\in C(X)$ (\emph{Seever's identity});
  \item $S(\varphi S\psi)=S\varphi S\psi$
    for all $\varphi,\psi\in C(X)$ ($S$ is an \emph{averaging operator}) iff $\mu_x$ is ergodic for all $x\in X$;
  \item $T_gS=S$ for all $g\in G$ iff each orbit closure is uniquely ergodic.
\end{enumerate}
\end{thm}
Here, $T_g \varphi=\varphi\circ g$ for each $g\in G$ and $\varphi\in C(X)$, see Theorem \ref{thm:rel_bounded_lin_op} for the details.
As an immediate consequence of this operator-theoretic perspective, we get that in general, the space of all invariant measures of a $\mc F$-weak mean 
equicontinuous action is a Bauer simplex, see Corollary \ref{cor3.16}.

\section{Preliminaries}\label{sec:preliminaries}

We start by briefly reviewing some basics from topological dynamics and ergodic theory.
For a thorough discussion of these topics from a functional analytical perspective (which we frequently assume in this note), see \cite{EisnerFarkasHaaseNagel2015}.

	 A \emph{topological dynamical system} (or simply a \emph{system}) is a triple $(X,G,\alpha)$ where $G$ is a topological group, $X$ is a compact metrizable space and $\alpha\colon G\times X\to X$ is a jointly continuous left action of $G$ on $X$.
	 We often just write $gx$ for $\alpha(g,x)$ and keep $\alpha$ implicit by simply referring to $(X,G)$ as a system.
	 For $x\in X$, we denote by $Gx=\{gx\: g\in G\}$ the \emph{orbit} of $x$.
	 	 
	 Given a system $(X,G)$, the action of $G$ on $X$ canonically defines the so-called \emph{Koopman representation} of $G$ on $C(X)$---the space of all continuous functions on $X$ equipped with the $\sup$-norm 
	 $\|\cdot\|_\infty$.
	 Specifically, given $f\in C(X)$ and $g\in G$, we write $g.f$ for the mapping $x\mapsto f(g x)$ in $C(X)$.
	 Through the dual of this representation, $G$ further acts on $\mathcal{M}(X)$---the collection of all Borel probability measures on $X$.
	 Specifically, given $\mu\in \mathcal{M}(X)$ and $g\in G$, we write $g^*\mu$ for the measure $\nu \in \mc M(X)$ with $\nu(f)=\mu(g.f)$ for $f\in C(X)$.
	 If $\mu\in \mc M(X)$ satisfies $g^*\mu=\mu$, we call $\mu$ \emph{invariant}.
	 The collection of all invariant measures of $(X,G)$ is denoted by $\mc M (X,G)$.
	 Recall that a measure $\mu \in \mc M(X,G)$ is \emph{ergodic} if each essentially invariant Borel set $A\ssq X$ satisfies $\mu(A)\in\{0,1\}$.
	 Here, $A$ is \emph{essentially invariant} if $\mu(A\triangle gA)=0$ for each $g\in G$ with $\triangle$ the symmetric difference.

\subsection{Wasserstein distance}	
We throughout consider $\mc M(X)$ and $\mc M(X,G)$ equipped with the 
weak*-topology.
Among the metrics which are compatible with this topology, we will make use of the \emph{Wasserstein distance}.

	Given $\mu,\nu\in \mathcal{M}(X)$, recall that $\iota\in \mathcal{M}(X^2)$ is a \emph{coupling} of $\mu$ with $\nu$,
	if the pushforwards of the projections $\pi_1$ and $\pi_2$ to the first and second coordinate, respectively, satisfy
	$\pi_1^*\iota=\mu$ and $\pi_2^*\iota=\nu$.
	The Wasserstein distance ${W}\colon \mathcal{M}(X)^2\to [0,\infty)$ is defined through 
	\[
	{W}(\mu,\nu)=\inf_\iota \iota(d),
	\]
	where the infimum is taken over all couplings $\iota$ of $\mu$ with $\nu$, and $d$ is some compatible metric on $X$.
	It is well known that
	${W}$ is a metric on $\mathcal{M}(X)$ which induces the weak*-topology, see \cite[Chapter~7]{Villani2003}.
	
	In fact, we will frequently utilize an alternative way of computing $W$ via
\[
{W}(\mu,\nu)=\sup_{f}\left|\int_Xfd\mu-\int_Xfd\nu\right|,
\]
where the supremum is taken over all Lipschitz functions $f$ on $X$ with $\|f\|_{Lip}\leq 1$ \cite[Remark~7.5]{Villani2003}. 
Here, $\|f\|_{Lip}$ denotes the infimum of all possible Lipschitz constants for $f$.

\subsection{F\o lner sequences}\label{sec:Folner_seq}
    In order to study statistical properties along the orbits of a system $(X,G)$, one often averages along F\o lner sequences.
    
    Besides the standard F\o lner sequences---which we may refer to as \emph{thick} F\o lner sequences in the following---we also deal with so-called \emph{thin} F\o lner sequences, which allow to deal with amenable groups which are not locally compact \cite{SchneiderThom2018}.
    Note that whenever we speak of F\o lner sequences (thin or thick), we actually refer to \emph{left} F\o lner sequences.
	At times, we will also deal with \emph{right} F\o lner sequences 
	but whenever we do, we explicitly mention the attribute \emph{right}.
	The definition of right F\o lner sequences is the same as for the left ones only that below---in \eqref{eq: defn thick Folner} and \eqref{eq: defn thin Folner}---one has to multiply from the right by $K$ and $g$, respectively.
	Further, in \eqref{eq: defn thick Folner}, the Haar measure $\theta$ has to be replaced by a right Haar measure $\theta_r$
	while in \eqref{eq: defn thin Folner}, the definition of $\mathfrak{m}_V$ has to be adjusted in the obvious way.

	\subsubsection{(Thick) F\o lner sequences}
	Let $G$ be a locally compact group with a (left) Haar measure $\theta$.
	A sequence $(F_n)_{n\in \mathbb{N}}$ of non-empty compact subsets in $G$ with positive Haar measure is called a 
  \emph{(thick) F\o lner sequence} if
	\begin{align}\label{eq: defn thick Folner}
	\lim_{n\to\infty} \frac{\theta(F_n\Delta KF_n)}{\theta(F_n)}= 0,
	\end{align}
  for all compact $K\subseteq G$.
	
	It is well-known that a locally compact group is $\sigma$-compact and amenable if and only if it contains a (thick) F\o lner sequence
	\cite[Theorem~3.2.1]{EmersonGreenleaf1967}.
	\begin{prop}\label{prop:thickFolner_combine}
	If $(\mathcal{F}^{(m)})_{m\in \mathbb{N}}$ is a countable family of thick F\o lner sequences in a locally compact topological group $G$, then there exists a thick F\o lner sequence $\mathcal{F}$ such that $\mathcal{F}$ and  $\mathcal{F}^{(m)}$ have a common subsequence for any $m\in \mathbb{N}$. 
\end{prop}
\begin{proof}
	As $G$ allows for a (thick) F\o lner sequence, it is $\sigma$-compact. 
	That is, there is a sequence $(K_n)_{n\in \mathbb{N}}$ of compact sets with $\bigcup_{n\in \mathbb{N}}K_n=G$ and $K_n\subseteq K_{n+1}$.  
	Let $(m_n)_{n\in \mathbb{N}}$ be a sequence in $\mathbb{N}$ such that for each $m\in \mathbb{N}$, we have $m_n=m$ infinitely often.
	We pick $F_n$ from $\mathcal{F}^{(m_n)}$ with
	$\theta(K_n F_n \Delta F_n)/\theta(F_n)\leq 1/n$.	
\end{proof}
	
\subsubsection{Thin F\o lner sequences}\label{sec: thin Folner sequence}

	Let $G$ be a topological group. 
 	Consider finite subsets $F,E\ssq G$ and a neighbourhood $V$ of the neutral element $e_G \in G$. 
		The \emph{$V$-matching number} $\mathfrak{m}_V(F,E)$ of $F$ and $E$ is the
    maximal cardinality of a subset $M$ of $F$ with an injection $\phi\colon M\to E$ such that for every $f\in M$, we have $\phi(f)\in Vf$.
		
	The following concept was introduced in \cite{SchneiderThom2018}. 
	A sequence $(F_n)_{n\in \mathbb{N}}$ of finite non-empty subsets of $G$ is said to be a \emph{thin F\o lner sequence} in $G$ if for every $g\in G$ and every open neighbourhood $V$ of $e_G$, we have 
	 \begin{align}\label{eq: defn thin Folner}
	 \lim_{n\to \infty}\frac{\mathfrak{m}_V(F_n,gF_n)}{|F_n|}=1.
	 \end{align}
	 \begin{rem}\label{rem:usualFolner}
	If $G$ is a countable discrete group, then the concepts of thin and thick
	F\o lner sequences coincide.
	\end{rem}
	
	\begin{rem}\label{rem: separable}
	Note that a topological group that allows for a thin F\o lner sequence is separable. Indeed, if $(F_n)_{n\in \mathbb{N}}$ is such a sequence, the countable set $\bigcup_{n\in \mathbb{N}} F_nF_n^{-1}$ is dense. 
\end{rem}
\begin{rem}
 As much as it is straightforward, it is important to observe that given a thin F\o lner sequence $(F_n)$ and some sequence $(g_n)$ in $G$, we have ${\mathfrak{m}_V(F_n,gF_n)}={\mathfrak{m}_V(F_ng_n,gF_ng_n)}$
 for each $g\in G$ and each neighbourhood $V$ of the neutral element.
 In particular, with $(F_n)$ also $(F_n g_n)$ is F\o lner.
\end{rem}

	 A second countable topological group $G$ is amenable if and only if it allows for a thin F\o lner sequence. 
	Indeed, in \cite[Remark 4.6]{SchneiderThom2018}, it is shown that whenever there is a thin F\o lner sequence, then $G$ is amenable---see also Corollary~\ref{cor: Krylov-Bogolyubov} below.
	To observe the converse, let $\{g_n\in G\: n\in \mathbb{N}\}$ be a countable dense subset of $G$ and let $V_n$ be the open ball of radius $1/n$ centred at the neutral element $e_G$. 
	With \cite[Theorem 4.5]{SchneiderThom2018}, we get that for each $n\in \mathbb{N}$, there exists a finite set $F_n\subseteq G$ such that 
	$\mathfrak{m}_{V_n}(F_n,g_kF_n)/|F_n|\geq 1-1/n$ for $k=1,\dots, n$. 
	It is straightforward to see that $(F_n)_{n\in \mathbb{N}}$ is a thin F\o lner sequence in $G$. 

	\begin{example}\label{exmp: hom+ on interval}
	 Examples of amenable not locally compact groups are given by \emph{extremely amenable groups} \cite{GiordanoPestov2002}.
	 These include the orientation-preserving homeomorphisms $\textrm{Hom}_+(\mathbb I)$ (equipped with the topology of uniform convergence) on the closed unit interval $\I$  \cite{Pestov1998}.
    By the above, there is hence a thin F\o lner sequence $(F_n)$ in $\textrm{Hom}_+(\mathbb I)$.
	\end{example}

Similarly to Proposition~\ref{prop:thickFolner_combine}, we have	
\begin{prop}\label{prop:thinFolner_combine}
	If $(\mathcal{F}^{(m)})_{m\in \mathbb{N}}$ is a countable family of thin F\o lner sequences in a second countable topological group $G$, then there exists a thin F\o lner sequence $\mathcal{F}$ such that $\mathcal{F}$ and  $\mathcal{F}^{(m)}$ have a common subsequence for any $m\in \mathbb{N}$. 
\end{prop}
\begin{rem}
Note that a topological group is second countable if and only if it is metrizable and separable. 
First, due to the Birkhoff-Kakutani Theorem, a topological group is metrizable if and only if it is first countable \cite[Section~1.22]{MontgomeryZippin2018}.
	Second, a metric space is separable if and only if it is second countable. 
\end{rem}
\begin{proof}[Proof of Proposition~\ref{prop:thinFolner_combine}]
	Let $\{g_n\in G\: n\in \mathbb{N}\}$ be dense in $G$ and let $(m_n)_{n\in \mathbb{N}}$ be a sequence in $\mathbb{N}$ such that for each $m\in \mathbb{N}$, we have $m_n=m$ infinitely often. 
	Denote by $V_n$ the open ball of radius $1/n$ centred at $e_G$. 
	We pick $F_n$ from $\mathcal{F}^{(m_n)}$ with
	$\mathfrak{m}_{V_n}(F_n,g_kF_n)/|F_n|\geq 1- 1/n$ for all $k=1,\dots,n$.
\end{proof}

\paragraph*{\textbf{Standing assumptions}}
    Throughout this work and without further mentioning, $G$ is assumed to be a
	$\sigma$-compact locally compact amenable group or a second countable amenable group.
    When speaking of a F\o lner sequence in $G$ (without mentioning thin or thick), we refer to whatever concept of F\o lner sequence (thin or thick) is available in $G$.
    In fact, for the sake of a concise presentation, we formulate most statements without specifying the kind of the involved F\o lner sequences and readers may interpret each such statement as two statements---one in which all the occurring F\o lner sequences are thin and another one in which all are thick.
 In a similar vein, given a (thin or thick) F\o lner sequence $\mc F =(F_n)_{n\in \N}$ in $G$ and some map $f$ from $G$ into $\R$, we may write the average of $f$ over $F_n$ as
 $1/\theta(F_n)\cdot \int_{F_n}\!f(g) \,d\theta(g)$ where $\theta$ denotes the counting measure in case $(F_n)$ is a thin F\o lner sequence while it denotes a Haar measure if $(F_n)$ is thick.
	
	A background on ergodic theory with (thin) F\o lner sequences
	is given in Appendix~\ref{appendix: thin folner}.

\section{Weak mean equicontinuity}\label{sec:weak_mean_equi}
	Given a dynamical system $(X,G)$ and a F\o lner sequence $\mathcal{F}=(F_n)_{n\in \mathbb{N}}$ in $G$, for each $n\in \N$, we define a positive bounded operator $S_n=S_n^\mathcal{F}\: C(X)\to C(X)$ through
	\[
	 S_n\:  f\mapsto\Big(x\mapsto 1/\theta(F_n)\int_{F_n}\!f(gx)\,d\theta(g)\Big).
	\]
	Following \cite{XuZheng2022}, based on the dual $S_n^*$ and the Wasserstein distance $W$, we introduce the pseudometrics
\begin{align*}
\mathcal{W}_\mathcal{F}(x,y)= \limsup_{n\to \infty}{W}(S_n^*\delta_x,S_n^*\delta_y)
\end{align*}
	and $\mathcal{W}(x,y)=\sup_{\mathcal{F}}\mathcal{W}_\mathcal{F}(x,y)$, where the supremum is taken over all F\o lner sequences. 
While most of the time, we are interested in left F\o lner sequences in this work, we may also consider the above objects (in particular, $S_n$ and $\mc W_{\mc F}$) to be defined for right F\o lner sequences.
Formally, the respective definitions remain the same after replacing the (left) Haar measure $\theta$ by a right Haar measure $\theta_r$.

\begin{rem}\label{rem: averaged Wasserstein is bigger than averaged narrow}
As in \cite[Appendix]{XuZheng2022}, one can see that if $\mc F$ is a thin F\o lner sequence in $G$, then
\begin{align}\label{eq: original defn of Wf}
\mathcal{W}_\mathcal{F}(x,y)= \limsup_{n\to \infty}\inf_{\sigma}\frac{1}{|F_n|}\sum_{g\in F_n}d(gx,\sigma(g)y),
\end{align}
where the infimum is taken over all permutations $\sigma$ of $F_n$.
Indeed, for $G=\Z$ and $\mathcal{F}$ the standard F\o lner sequence in $\Z$, \eqref{eq: original defn of Wf}
was the original definition for the pseudometric $\mathcal{W}_\mathcal{F}$ as
introduced in \cite{ZhengZheng2020} (motivated by the idea that statistical properties of the long time behaviour should not depend on the dynamical order on orbits).
Shortly after, the above equation was established in \cite{CaiKwietniakLiPourmand2022}
for countable abelian groups and in \cite{XuZheng2022} for countable discrete 
amenable groups.
\end{rem}

In the present work, we study the relationship between the continuity of the above pseudometrics and the continuity of F\o lner averages.

\begin{definition}
 We say $(X,G)$ is \emph{$\mc F$-weakly mean equicontinuous} if
$\mathcal{W}_\mathcal{F}\in C(X^2)$.
We call $(X,G)$ \emph{weakly mean equicontinuous} if $\mathcal{W}\in C(X^2)$.
\end{definition}
The next statement shows that this terminology is in line with \cite{XuZheng2022}.

\begin{prop}\label{pro:WeylweakmeanequicontinuityviaallFolner}
	A system $(X,G)$ is weakly mean equicontinuous if and only if $\mathcal{W}_\mathcal{F}\in C(X^2)$ for every F\o lner sequence $\mathcal{F}$. 
\end{prop}

\begin{proof}
	Since $\mathcal{W}(x,y)\geq \mathcal{W}_\mathcal{F}(x,y)$, one implication is obvious. 
	For the converse, suppose $\mathcal{W}_\mathcal{F}$ is continuous for all F\o lner sequences $\mathcal{F}$. 
	Let $(x_n)_{n\in \mathbb{N}}$ be a sequence in $X$ with $x_n\to x\in X$. 
	Since $\mathcal{W}$ is a pseudometric, it suffices to show that $\mathcal{W}(x_n,x)\to 0$. 
	For $n\in\mathbb{N}$, there exists a F\o lner sequence $\mathcal{F}^{(n)}$ such that $\mathcal{W}_{\mathcal{F}^{(n)}}(x_n,x)+1/n\geq \mathcal{W}(x_n,x)$ and such that $\mathcal{W}_{\mathcal{F}^{(n)}}(x_n,x)$ is a limit. 
	By Propositions~\ref{prop:thickFolner_combine} and \ref{prop:thinFolner_combine} (in the thick and thin case, respectively), there exists a F\o lner sequence $\mathcal{F}$ such that $\mathcal{F}$ and $\mathcal{F}^{(n)}$ have a common subsequence $\tilde {\mathcal{F}}^{(n)}$ for all $n\in \mathbb{N}$.
	Thus 
	$\mathcal{W}(x_n,x)
	\leq \mathcal{W}_{\mathcal{F}^{(n)}}(x_n,x)+1/n
	= \mathcal{W}_{\tilde{\mathcal{F}}^{(n)}}(x_n,x)+1/n
	\leq \mathcal{W}_{\mathcal{F}}(x_n,x)+1/n
	\to 0$.  	
\end{proof}

\begin{thm}\label{thm: characterisations of weak mean equicontinuity}
	Let $(X,G)$ be a topological dynamical system.
	Suppose there is a left or right F\o lner sequence $\mc F$ such that
\begin{itemize}
	\item[(i)] 
	$(X,G)$ is $\mc F$-weakly mean equicontinuous.
\end{itemize}	
Then there is a subsequence $\mc F'$ of $\mc F$ such that
\begin{itemize}
    \item[(ii)] Every $x\in X$ is $\mu_x$-generic along $\mathcal{F}'$ for some measure $\mu_x\in \mathcal{M}(X)$ and the map $X\ni x\mapsto \mu_x\in \mathcal{M}(X)$ is continuous. 
    \item[(iii)] There is a bounded linear operator $S\: C(X) \to C(X)$ such that $S_n^{\mc F'}f(x)$ converges to $Sf(x)$ for each $f\in C(X)$ and $x\in X$.
\end{itemize}
	Moreover, (ii) holds (for some $\mc F'$) if and only if (iii) holds
	(for the same sequence $\mc F'$) and further, (ii) and (iii) imply (i) with $\mc F=\mc F'$.
\end{thm}
\begin{rem}
 Recall that $x\in X$ is \emph{$\mu$-generic along $\Fol$} if for each $f\in C(X)$, the sequence $S_n(f)(x)$ converges to $\mu(f)$.
 Clearly, if $x$ is $\mu$-generic along some F\o lner sequence, then $\supp(\mu)\ssq \overline {Gx}$.
\end{rem}

\begin{proof}
    For (i) $\Rightarrow$ (ii), given $n\in \N$, let $X_n\ssq X$  be some finite set which is $1/n$-dense in $X$.
    Due to compactness of $\mc M(X)$, 
    we can recursively define for each $n\in \N$ a subsequence $\mc F^{(n)}$ of $\mc F^{(n-1)}$ (starting with $\mc F^{(0)}=\mc F$) such that
    each $x\in X_n$ is $\mu_x$-generic along $\mc F^{(n)}$ for some $\mu_x\in \mc M(X)$.
    A diagonal argument then gives a subsequence $\mc F'$ of $\mc F$ such that each $x$ in the dense set $\tilde X=\bigcup_{n\in \N} X_n$ is $\mu_x$-generic along $\mc F'$ for some $\mu_x\in \mc M(X)$.
    
    Now, as $\mc W_{\mc F'}$ is continuous (since we assume $\mc W_{\mc F}$ to be continuous), we have
	\begin{align}\label{eq: Sn on Lipschitz functions}\limsup_{n\to \infty}|S_n^{\mc F'}f(x)-S_n^{\mc F'}f(y)|\leq\limsup_{n\to \infty}W({S_n^{\mc F'}}^*\delta_x,{S_n^{\mc F'}}^*\delta_y)=\mathcal{W}_\mathcal{F'}(x,y)
	\end{align}
	for each Lipschitz function $f$ with $\|f\|_{Lip}\leq 1$ and all $x,y\in X$.
	
	As a consequence, given such Lipschitz function $f$, some $y\in X$ and $\eps>0$, we can pick 
	$x\in \tilde X$ with $\mathcal{W}_{\mc F'}(x,y)<\eps$ to obtain that for all $n,m>N$ (with $N$ such that 
	for $n,m>N$ we have
	$|S_n^{\mc F'}f(x)-S_m^{\mc F'}f(x)|<\eps$
	and $|S_n^{\mc F'}f(x)-S_n^{\mc F'}f(y)|\leq\mathcal{W}_{\mathcal F'}(x,y)+\eps$)
	\begin{align*}
	 &|S_n^{\mc F'}f(y)-S_m^{\mc F'}f(y)|\\
	 &\leq 
	 |S_n^{\mc F'}f(y)-S_n^{\mc F'}f(x)|+
	 |S_n^{\mc F'}f(x)-S_m^{\mc F'}f(x)|+
	 |S_m^{\mc F'}f(x)-S_m^{\mc F'}f(y)|\leq 5\eps.
	\end{align*}
    In other words, $S_n^{\mc F'}f(y)$
    converges for all $y\in X$.
    It follows that for any function $g$ in the span of the Lipschitz functions $f$ with $\|f\|_{Lip}\leq 1$,    
    $S_n^{\mc F'}g(y)$ converges for each $y\in X$.
    The collection of all such $g$ is dense in $C(X)$ (see \cite[Chapter 12]{Carothers2000}) and it is not hard to see
    that hence, $S_n^{\mc F'}f(x)$ converges for each $f\in C(X)$ and $x\in X$.
    We write $\mu_x(f)=\lim_{n\to \infty}S_n^{\mc F'}f(x)$.
	Clearly, $\mu_x\: C(X)\ni f\mapsto \mu_x(f)$ lies in $\mc M(X)$.
	Further, the map $x\mapsto \mu_x$ is continuous since $(X,G)$ is $\mc F'$-weakly mean equicontinuous.

	For (ii) $\Rightarrow$ (i), let $\mathcal{F}'$ be a F\o lner sequence such that (ii) holds. 
	Then, for $x,y\in X$, we have 
	$S_n^*\delta_x\rightarrow \mu_x$ and 
	$S_n^*\delta_y\rightarrow \mu_y$. 
	Thus, $\mathcal{W}_\mathcal{F'}(x,y)=\limsup_{n\to \infty}W(S_n^*\delta_x,S_n^*\delta_y)=W(\mu_x,\mu_y)$, which is continuous in $(x,y)\in X^2$. 	

	We next show that (ii) (for a given F\o lner sequence $\mc F'$) implies (iii) (with the same F\o lner sequence).
    To that end, set $S\: C(X)\to C(X), \, f\mapsto \mu_{(\cdot)}(f)$ with $\mu_{(\cdot)}$ given by (ii).
    Due to (ii), $S$ is well defined.
    It is straightforward to see that $S$ is bounded and linear.
    Moreover, due to (ii), for every $f\in C(X)$,
    $S_n^{\mc F'}f$ converges pointwise to $Sf$.
	
	Finally, for (iii) $\Rightarrow$ (ii), we define $\mu_x=S^*\delta_x$ for $x\in X$. 
	Clearly, the mapping $x\mapsto\mu_x$ is continuous. 
	Further,
	\[S_n^{\mc F'} f(x)\rightarrow Sf(x)=(S^*\delta_x)(f)=\mu_x(f).\]
	We hence conclude that $x$ is $\mu_x$-generic along $\mc F'$.
\end{proof}

\begin{rem}\label{rem: invariance of measures in Thm 34}
	In the above statement, if $\mc F'$ is a left
	F\o lner sequence, then $\mu_x\in \mc M(X,G)$ by a standard 
	Krylov-Bogolyubov argument.
	We will see later (last part of Theorem \ref{the:Weyl-weak-mean}) that the same holds
	true if $\mc F'$ is a right F\o lner sequence.
\end{rem}

\begin{cor}\label{cor:operator_pos_cont_projection}
	Suppose $(X,G)$ is $\mc F$-weakly mean equicontinuous for some F\o lner sequence $\mc F$.
	Then the operator $S:C(X)\to C(X)$ from Theorem~\ref{thm: characterisations of weak mean equicontinuity}
	is a positive contractive projection, that is, 
	$Sf\geq 0$ if $f\geq 0$, $S\boldsymbol{1}\leq \boldsymbol 1$ and $S^2=S$.
\end{cor}
\begin{proof}
	Recall that $S$ is the limit of a subsequence of $(S_n^{\mc F})_{n\in\N}$.
	For notational convenience, we denote this subsequence simply by $(S_n)$.
	The positivity and $S\boldsymbol{1}\leq \boldsymbol 1$ are straightforwardly inherited from the elements of $(S_n)$.
	
	To see that $S$ is a projection, first observe
	\[
		(S_n^*\mu)f=\mu(S_n f)\xrightarrow{n\to\infty}\mu(S f)=(S^*\mu)f\quad\textnormal{for all}\quad \mu\in\mc M(X), f\in C(X),
	\]
	using $\lim_{n\to\infty}S_n f(x)=S f(x)$ for all $x\in X$ and 
	dominated convergence.
	Accordingly, by a standard Krylov-Bogolyubov argument, $S^*\mu\in\mc M(X,G)$.
	Now, due to the invariance of $S^*\mu$, we have $S_n^*(S^*\mu)=S^*\mu$ for each $n\in\N$. 
	This in turn gives for $\mu\in\mc M(X)$ and $f\in C(X)$
	\[
		(S^*(S^*\mu))f=\lim_{n\to\infty}(S_n^*(S^*\mu))f=(S^*\mu)f,
	\]
	that is, $S^*$ is a projection.
	Finally, due to the Hahn-Banach Theorem, $\mc M(X)$ separates points in $C(X)$ so that the above implies $S^2=S$.
\end{proof}

\begin{example}\label{exmp: F-weak on interval}
We next describe a F\o lner sequence $\hat {\mc F}$ with respect to which $(\I,G)$---where $G=\textrm{Hom}_+(\mathbb I)$ acts on $\I$ in the obvious way---is $\hat {\mc F}$-weakly mean equicontinuous.

To that end, recall from Example~\ref{exmp: hom+ on interval} that there actually is some thin F\o lner sequence $(F_n)$ in $G$.
Let $(h_n)$ be a dense sequence in $G$.
In the following, we may assume without loss of generality (by possibly going over to a subsequence) that
$\frak{m}_{B_{1/n}(\Id)}(F_N,h_nF_N)/|F_N|\geq 1-1/N$ whenever $N\geq n$,
where $B_{1/n}(\Id)$ is the $1/n$-ball centred at the identity $\Id$.

Now, given $x\in \I$ and $\eps>0$, let us refer to $g\in G$ as \emph{$(x,\eps)$-repelling} if $gy<\eps$ whenever $y<x-\eps$ and $gy>1-\eps$ whenever $y>x+\eps$.
Note that for each $x\in  \I$ and each $n\in \N$, there is $g_n^x\in G$ such that each element in $F_ng_n^x$ is $(x,1/n^2)$-repelling.
Indeed, with $\delta>0$ such that $g\delta<1/n^2$ and $g(1-\delta)>1-1/n^2$ for all $g\in F_n$, we may choose
$g_n^x$ to be any $(x,\min\{\delta,1/n^2\})$-repelling element in $G$.

Define $\hat {\mc F}=(\hat F_n)$ by $\hat F_n= \bigcup_{x=0,\frac{1}n,\ldots,1}F_ng_n^x$.
Then $\frak{m}_{B_{1/n}(\Id)}(\hat F_N,h_n \hat F_N)/|\hat F_N|\geq 1- 1/N$ if $N\geq n$ and it is easy to see that $\hat {\mc F}$ is hence F\o lner.
Further, for each $y\in \I$, we have
\begin{align*}
 \frac{\{g\in \hat F_n \: gy \leq 1/n^2\}}{|\hat {F_n}|}\to 1-y
 \ \qquad \text{and} \qquad \
 \frac{\{g\in \hat F_n \: gy \geq 1- 1/n^2\}}{|\hat {F_n}|}\to y.
\end{align*}
In other words, for each $f\in C(\I)$ and each $y\in \I$, we have
$S^{\hat {\mc F}}f(y)=(1-y)f(0)+yf(1)$, that is, $y$ is $\mu_y$-generic
along $\hat {\mc F}$, with $\mu_y=(1-y)\delta_0+y\delta_1$.
\end{example}

\begin{lem}\label{lem: unique ergodicity support ergodic measure}
	Let $(X,G)$ be a system and let $\mc F$ be a F\o lner sequence in $G$.
	If $(X,G)$ is $\mc F$-weakly mean equicontinuous, then the support of each ergodic measure is uniquely ergodic.
\end{lem}
\begin{proof}
	Consider two ergodic measures $\mu,\nu \in \mc M(X,G)$ and assume without loss of generality that
	$\operatorname{supp}(\nu)\subseteq \operatorname{supp}(\mu)$. 
	By Corollary~\ref{cor: almost all points are generic}, there exists a subsequence $\mathcal{F}'$ of $\mathcal{F}$ such that $\mu$-almost every point is $\mu$-generic along $\mc F'$ and $\nu$-almost every point is $\nu$-generic along $\mathcal{F}'$.
	In particular, there is $y\in X$ which is $\nu$-generic along $\mathcal{F}'$ and a sequence $(x_n)_{n\in \mathbb{N}}$ where each $x_n$ is $\mu$-generic along $\mathcal{F}'$ and $x_n\to y$. 
	Hence,
	$\mathcal{W}_\mathcal{F}(x_n,y)\geq \mathcal{W}_{\mathcal{F}'}(x_n,y)=W(\mu,\nu)$
	and the continuity of $\mathcal{W}_\mathcal{F}$ implies $\mu=\nu$. 
\end{proof}

For the convenience of the reader, we include a proof of the next statement; for $\Z$-actions, see also \cite[Proposition 3.9]{Furstenberg1981}.

\begin{lem}\label{lem:Furstenberglemma}
Let $(X,G)$ be a system.
For every ergodic measure $\mu$ and every transitive point $x\in X$, there is a
F\o lner sequence $\mathcal{F}$ such that $x$ is $\mu$-generic along 
$\mathcal{F}$. 	
\end{lem}
\begin{rem}
 Recall that $x\in X$ is \emph{transitive} if $\overline{Gx}=X$.
\end{rem}

\begin{proof}
Let $(f_n)_{n\in \mathbb{N}}$ be dense in $C(X)$. 
	By Corollary~\ref{cor: almost all points are generic}, there is a F\o lner sequence $(F_n)_{n\in \mathbb{N}}$ and a point $y\in X$ such that $y$ is $\mu$-generic along $(F_n)_{n\in \mathbb{N}}$. 
	 By possibly restricting to a subsequence, we may assume without loss of generality that 
	$\left|1/\theta(F_n)\int_{F_n}\!f_j(gy)\, d\theta(g)-\mu(f_j)\right|\leq 1/(2n)$
	for all $j=1,\dots,n$. 
	For $n\in \mathbb{N}$, there exists $\eps_n>0$ such that for all $z\in B_{\eps_n}(y)$ and $j\in \{1,\dots,n\}$, we have 
	$\left|1/\theta(F_n)\int_{F_n}\!f_j(gz)\, d\theta(g)-\mu(f_j)\right|\leq 1/n$.
	By transitivity of $x$, there exists $g_n\in G$ such that $g_nx\in B_{\eps_n}(y)$. 
	Therefore,
	\[
	\left|\frac{1}{\theta(F_ng_n)}\int_{F_ng_n}f_j(gx)\, d\theta(g)-\mu(f_j)\right|=\left|\frac{1}{\theta(F_n)}\int_{F_n}f_j(gg_nx)\, d\theta(g)-\mu(f_j)\right|\leq \frac{1}{n}.
	\]
	This shows that $x$ is $\mu$-generic along the F\o lner sequence $(F_ng_n)_{n\in \mathbb{N}}$.
\end{proof}

The next assertion generalizes Theorems 3.5 and 4.3 in \cite{XuZheng2022}
beyond the setting of countable discrete amenable groups.

\begin{thm}\label{the:Weyl-weak-mean}
	Let $(X,G)$ be a system. 
	The following statements are equivalent. 
\begin{itemize}
	\item[(i)] $(X,G)$ is weakly mean equicontinuous.
	\item[(ii)] For all $x\in X$, the orbit closure $\overline{Gx}$ is uniquely
    ergodic with an ergodic measure $\mu_x$ and the map $x\mapsto \mu_x$ is continuous. 
	\item[(iii)] For every F\o lner sequence $\mc F$ in $G$, $S_n^{\mc F}$ converges in the strong operator topology.
\end{itemize}	
Moreover, if $(X,G)$ is $\mc F$-weakly mean equicontinuous for some F\o lner sequence $\mc F$ and
there is an invariant measure $\mu$ with full support, that is, $\supp(\mu)=X$, then
(i)--(iii) are satisfied.

Finally, if $\mathcal{W}_{\mathcal{F}}$ is continuous
for some right F\o lner sequence $\mc F$ in $G$, then (i)--(iii) are satisfied. 
\end{thm}
\begin{proof}
	Note that (iii) $\Rightarrow$ (i) follows from Theorem~\ref{thm: characterisations of weak mean equicontinuity} (in combination with  Proposition~\ref{pro:WeylweakmeanequicontinuityviaallFolner}).
	
	We first discuss (i) $\Rightarrow$ (ii).
	Due to Theorem~\ref{thm: characterisations of weak mean equicontinuity} and Remark~\ref{rem: invariance of measures in Thm 34}, it suffices to show that for every $x\in X$, the orbit closure $\overline{Gx}$ is uniquely ergodic. 
	To that end, consider $\mu,\nu\in \mathcal{M}(\overline{Gx},G)$. 
	Let $\eps>0$, $y\in \operatorname{supp}(\mu)$ and pick $z\in {Gx}$ such that $\mathcal{W}(y,z)<\eps$. 
	By Lemma~\ref{lem:Furstenberglemma}, there exists a F\o lner sequence $\mathcal{F}$ such that $z$ is $\nu$-generic along $\mathcal{F}$. 
	By Corollary~\ref{cor: almost all points are generic}, there exists a subsequence $\mathcal{F}'$ of $\mathcal{F}$ such that $\mu$-almost every point is $\mu$-generic along $\mc F'$. 
	Accordingly, there is $y'\in \operatorname{supp}(\mu)$ such that $y'$ is $\mu$-generic along $\mathcal{F}'$ and $\mathcal{W}(y,y')<\eps$. 
	Now,
	\[
	2\eps\geq \mathcal{W}(y',y)+\mathcal{W}(y,z)\geq\mathcal{W}(y',z)\geq \mathcal{W}_{\mathcal{F}'}(y',z)=W(\mu,\nu),
	\] 
	where we used that $W$ is compatible with the weak*-topology in the last step.
	As $\eps>0$ was arbitrary, we conclude $\mu=\nu$. 
  
  For (ii) $\Rightarrow$ (iii), we proceed in a similar spirit as in the proof of \cite[Theorem~4.4]{CaiKwietniakLiPourmand2022}.
  First, note that for each $f\in C(X)$,
  and each F\o lner sequence $\mc F$, Theorem~\ref{thm: uniform ergodic theorem} and point (ii) give that $x\mapsto S_n^{\mc F}f(x)=S_nf(x)$ converges pointwise to the continuous function $x\mapsto \mu_x(f)$.
  Note that strong convergence of $S_n$ is equivalent to uniformity of this convergence (for each $f$). 
  Hence, we assume
  for a contradiction that there is $f\in C(X)$ and $\eps>0$ such that for all $N\in \N$, there is $n\geq N$ and $x_n\in X$ with $|S_nf(x_n)-\mu_{x_n}(f)|>\eps$.
  By possibly going over to a subsequence, we may assume without loss of generality that $S_{n}^*\delta_{x_n}$ converges to some $\nu \in \mc M(X,G)$ (using Krylov-Bogolyubov) and that $x_n$ converges to some $x\in X$.
  Note that $|\nu(f)-\mu_{x}(f)|\geq \eps$.
  
  The ergodic decomposition of $\nu$ reads $\nu=\int_{\mc M(X,G)}\!\mu\,d\lambda(\mu)$ (see e.g.\ \cite[page 77]{Phelps2001}) where $\lambda$-almost every measure $\mu$ is ergodic and necessarily satisfies $\supp(\mu)\ssq \supp(\nu)$.
  Among those measures, there must be some $\mu_0$ with $|\mu_0(f)-\mu_{x}(f)|\geq \eps$.
  Pick some $y\in \supp(\mu_0)$ and observe that by (ii), $\mu_0=\mu_y$.

  Finally, by the Portmanteau Theorem (and since $y\in \supp(\nu)\supseteq \supp(\mu_y)$), for every open neighbourhood $U$ of $y$, we have
  \[
	\liminf_{n\to\infty}S_{n}^*\delta_{x_n}(U)=\liminf_{n\to\infty} 1/\theta (F_n)\cdot \theta(\{gx_n \in U\: g\in F_n\})\geq \nu(U)>0.
  \]
    As a consequence, there is a sequence $g_n$ in $G$ with $g_nx_n\to y$ so that, due to the continuity of $x\mapsto \mu_x$, we have $\mu_{g_nx_n}\to\mu_y$.
    At the same time, due to the unique ergodicity of orbit closures, $\mu_{g_nx_n}=\mu_{x_n}\to \mu_x$.
    It follows that $\mu_x=\mu_y$ in contradiction to the assumptions on $\mu_y=\mu_0$.

  To see the ``moreover''-part, given $x\in X$, observe that due to \cite[Lemma~6]{Farrell1962} (or, alternatively, the ergodic decomposition of $\mu$),
  there is a sequence $(\mu_n)$ of ergodic measures in $\mc M(X,G)$ and a sequence $(x_n)$ in $X$ with $x_n\in \supp(\mu_n)$ such that $x_n\to x$.
  By Lemma~\ref{lem: unique ergodicity support ergodic measure}, $\mc W_{\mc F}(x_n,gx_n)=0$ for each $g\in G$ so that continuity of $\mc W_{\mc F}$ gives
  \begin{align}\label{eq: constant on orbit}
   \mc W_{\mc F}(x,gx)=\lim_{n\to\infty}\mc W_{\mc F}(x_n,gx_n)=0.
  \end{align}
  Due to Theorem~\ref{thm: characterisations of weak mean equicontinuity}, we may assume without loss of generality
  that for each $f\in C(X)$, $y\mapsto S^{\mc F}f(y)=\lim_{n\to\infty}S^{\mc F}_nf(y)=\mu_y(f)$ is well-defined and continuous.
  With \eqref{eq: constant on orbit}, this gives that 
  $y\mapsto S^{\mc F}f(y)$ is actually constant on $\overline{Gx}$.
  Now, Theorem~\ref{thm: uniform ergodic theorem} and Tietze's Extension Theorem give that $\overline{Gx}$ is uniquely ergodic.
  As $x$ was arbitrary, this shows (ii).

  To see the last part,
  recall from Theorem~\ref{thm: characterisations of weak mean equicontinuity} that there is a bounded linear operator $S^{\mc F'}\: C(X)\to C(X)$ and a subsequence $\mc F'$ of $\mc F$ such that $S_n^{\mc F'}f(x)$ converges to $S^{\mc F'}f(x)$ for each $f\in C(X)$ and $x\in X$.
  Observe that since $\mc F'$ is right F\o lner, we have for each $f\in C(X)$ and all $g\in G$ that $S^{\mc F'}f(x)=S^{\mc F'}f(gx)$.
  As $S^{\mc F'}f$ is continuous, we therefore have that
  $S^{\mc F'}f$ is constant on orbit closures.
  With Proposition~\ref{prob:right Foelner implies unique ergodicity} and Tietze's Extension Theorem, this implies (ii) (and $S^{\mc F'}f(x)=\mu_x(f)$ for each $x\in X$).
\end{proof}

\begin{example}
As a consequence of the above, 
a transitive system is weakly mean equicontinuous if and only if it is uniquely ergodic.
In particular, the system $(\I,G)$ from Example~\ref{exmp: F-weak on interval} is not weakly mean equicontinuous.
However, the canonical action of $G=\textrm{Hom}_+(\I)$ on $\R/\Z$---obtained by identifying $0$ and $1$---clearly is weakly mean equicontinuous.
\end{example}

We end this section by combining the above results with some aspects of the general theory of positive contractive projections and averaging operators, see 
\cite[Section 11.3]{Rao2005} as well as \cite{Kelley1958,Seever1966,Hadded2007} (and references therein).

\begin{thm}\label{thm:rel_bounded_lin_op}
Let $\mc F$ be a left or right F\o lner sequence and suppose that $(X,G)$ is $\mc F$-weakly mean equicontinuous.
  Then there exists a subsequence $(F_n)$ of $\mc F$ and a bounded 
  linear operator $S\:C(X)\to C(X)$ such that for all $f\in C(X)$ and $x\in X$
  \begin{equation}\label{eq:conv_to_dual}
    \lim_{n\to\infty}1/\theta(F_n)\cdot \int_{F_n}\!f(gx)\, d\theta(g)=Sf(x).
  \end{equation}
  Vice versa, if there is a Følner sequence $(F_n)$ and a bounded linear
  operator $S\:C(X)\to C(X)$ satisfying $\eqref{eq:conv_to_dual}$, then $(X,G)$ is 
  $(F_n)$-weakly mean equicontinuous.
  Moreover, $S$ is a positive contractive projection and the following statements hold.
  \begin{enumerate}[(a)]
    \item $S(f Sh)=S(Sf Sh)$ for all $f,h\in C(X)$ (\emph{Seever's identity}).
    \item $S(f Sh)=Sf Sh$ for all $f,h\in C(X)$ ($S$ is an \emph{averaging operator}) if and only if $\mu_x$ is ergodic for all $x\in X$.
    \item $T_gS=S$ for all $g\in G$ (where $T_g f=g.f$ for $g\in G$ and $f\in C(X)$) if and only if each orbit closure is uniquely ergodic.
  \end{enumerate}
\end{thm}
\begin{proof}
	The first part is Theorem~\ref{thm: characterisations of weak mean equicontinuity}.
	Further, if $\mc F$ is left F\o lner, then
	Corollary~\ref{cor:operator_pos_cont_projection} gives that
	$S$ is a positive contractive projection.
	If, alternatively, $\mc F$ is right F\o lner, then
	it is again obvious that $S$ is 
	positive and contractive (as the expression on the left in \eqref{eq:conv_to_dual} is positive and contractive for each $n$).
	Furthermore, as shown in the last part of the proof of Theorem~\ref{the:Weyl-weak-mean},
	$Sf(x)=\mu_x(f)$ for all $x\in X$ and $f\in C(X)$ with
	$\mu_x\in\mc M(X,G)$ and $x\mapsto\mu_x$ constant on orbit closures.
	This immediately implies $S^2=S$, that is, $S$ is also a positive
	contractive projection for right F\o lner sequences.
	
	Now, positive contractive projections
	on $C(X)$ always fulfil Seever's identity, see \cite[Theorem~1]{Seever1966}.
	This gives item (a).
	
	For item (c), we make use of the fact that $Sf(x)=\mu_x(f)$ for some $\mu_x\in \mc M(X,G)$ (see Theorem~\ref{thm: characterisations of weak mean equicontinuity} and Remark~\ref{rem: invariance of measures in Thm 34} for left F\o lner sequences or again, the last part of the proof of Theorem~\ref{the:Weyl-weak-mean} for right F\o lner sequences).
	Let us first assume that each orbit closure is uniquely
	ergodic.
	Then, for all $f\in C(X)$ and $x\in X$ we get that $T_g Sf(x)=\mu_{gx}(f)=\mu_x(f)=Sf(x)$.
	Note that in the other direction,
	we only have to consider left F\o lner sequences due to the last part of Theorem~\ref{the:Weyl-weak-mean}.
	Now, if $S$ is invariant under all $T_g$'s, we have 
	that $\mu_{gx}(f)=T_g Sf(x)=Sf(x)=\mu_x(f)$ for all $f\in C(X)$, $x\in X$ and $g\in G$.
	That is, $x\mapsto \mu_x$ is constant along orbits.
	As furthermore, $x\mapsto \mu_x$ is continuous (Theorem~\ref{thm: characterisations of weak mean equicontinuity}), we have that 
	$x\mapsto \mu_x$ is constant on orbit closures and we obtain unique ergodicity as in the proof of 
	Theorem~\ref{the:Weyl-weak-mean} (using Corollary~\ref{cor: almost all points are generic}).
	
	Finally, to prove (b), we make use of the following characterization, see \cite[Theorem~2.2]{Kelley1958}:
	the operator $S$ is averaging if and only if for each $x\in X$ the support of $\mu_x$
	is contained in the set 
	\[
		D_x=\{y\in X: Sf(y)=Sf(x)\;\textnormal{for all}\;f\in C(X)\}
		=\{y\in X: \mu_y=\mu_x\}.
	\]
	First, observe that if each orbit closure is uniquely ergodic, then (b) is trivial
	since in this case, $\mu_x$ is obviously ergodic and
	$\supp(\mu_x)\subseteq\overline{Gx}\subseteq D_x$.
	Due to the last part of Theorem~\ref{the:Weyl-weak-mean}, we are hence left to show (b) only for left F\o lner sequences.
	
	Now, according to  Lemma \ref{lem: unique ergodicity support ergodic measure}, 
	we have $\supp(\mu_x)\subseteq D_x$ in case that $\mu_x$ is ergodic.
	Accordingly, $S$ is averaging if all $\mu_x$ are ergodic.
	For the other direction, assume that $\supp(\mu_x)$ is contained in $D_x$.
	Since $\supp(\mu_x)$ is closed and invariant, there is an
	ergodic measure $\mu$ with $\supp(\mu)\subseteq\supp(\mu_x)$.
	By Corollary~\ref{cor: almost all points are generic}, there exists a
	subsequence $(F_n')$ of $(F_n)$ such that $\mu$-almost 
	every point is $\mu$-generic along $(F_n')$.
	In particular, there exists $y\in\supp(\mu)$ with 
	\[
		\mu(f)=\lim_{n\to\infty}S_n^{(F_n')}f(y)=\lim_{n\to\infty}S_n^{(F_n)}f(y)=\mu_y(f)=\mu_x(f),
	\]
	for all $f\in C(X)$.
	Accordingly, all $\mu_x$ are ergodic.
\end{proof}

\begin{rem}
	Note that the relation $T_g S=S$ (for $g\in G$) in Theorem~\ref{thm:rel_bounded_lin_op}~(c)
	is equivalent to $T_g S=S T_g$ (for $g\in G$) since $\mu_x\in\mc M(X,G)$
	(for $x\in X$) is equivalent to $ST_g=S$ (for $g\in G$).
\end{rem}

Recall that $\mc M(X,G)$ is a simplex and that a simplex whose extreme points
form a closed set is called a \emph{Bauer simplex}, see \cite{Phelps2001}.
We can conclude from Theorem~\ref{thm:rel_bounded_lin_op} together with 
\cite[Theorem 3]{Lloyd1963} the following rigidity result which is also
indirectly contained in \cite{CaiKwietniakLiPourmand2022,XuZheng2022} for
countable discrete abelian/amenable groups and $\Fol$ a two-sided 
F\o lner sequence.

\begin{cor}\label{cor3.16}
	Let $(X,G)$ be a topological dynamical system which is $\mc F$-weakly mean equicontinuous
	for some left or right F\o lner sequence $\mc F$.
	Then $\mc M(X,G)$ is a Bauer simplex. 
\end{cor}

\section{$\mc F$-weak mean equicontinuity versus weak mean equicontinuity}\label{sec: example}

In Example~\ref{exmp: F-weak on interval}, we already saw that in general, $\mc F$-weak mean 
equicontinuity for some F\o lner sequence $\mc F$ does not imply weak mean equicontinuity.
As much as this phenomenon somehow appears to be related to the \emph{size} of the acting group, it is not merely a consequence of the lack of local compactness.
In fact, in this section, we give a complete description of the F\o lner averages of an action of the Lamplighter group---which, in contrast to $\textrm{Hom}_+(\I)$, is locally compact.
Among others, we will see that even if a system is $\mc F$-weakly mean equicontinuous and every point is generic for an \emph{ergodic} measure along $\mc F$, the system does not need to be weakly mean equicontinuous.

	Consider $\hat{X}=(\mathbb{Z}\cup\{ \infty\})\times \{1\}$ and $\check{X}=(\mathbb{Z}\cup \{\infty\})\times \{0\}$ and
	write $\hat{s}=(s,1)$ and $\check{s}=(s,0)$ for $s\in \mathbb{Z}\cup\{\infty\}$. 
	We equip $\hat{X}$ with some metric $\hat{d}$ 
	which induces the topology of the one-point compactification (of $\Z$) on $\hat X$
	and metrize $\check{X}$ similarly (by $\check{d}$).
	Finally, we equip $X=\check{X}\cup \hat{X}$ with the metric $d$ that restricts to $\check{d}$ and $\hat{d}$ on $\check{X}$ and $\hat{X}$, respectively, and which satisfies $d(x,y)=1$ for $x\in \hat{X}$ and $y\in \check{X}$.

	Let $f\colon X\to X$ be the transposition $f=(\hat{0},\check{0})$, that is, $f$ is the identity on $X\setminus \{\check{0},\hat{0}\}$ and satisfies $f(\hat{0})=\check{0}$, $f(\check{0})=\hat{0}$. 
	Let $\sigma\colon X\to X$ be the shift map with $\sigma(\check {s})=\widecheck{(s-1)}$ and $\sigma(\hat{s})=\widehat{(s-1)}$, where $\infty - 1=\infty$. 
	Then the countable discrete group $G=\langle f,\sigma\rangle$ is (isomorphic to) the Lamplighter group and acts on $X$ by homeomorphisms. 
	Defining $f_n=(\hat{n},\check{n})$, observe that 
	$f_n=\sigma^{-n}\circ f \circ \sigma^n$ and
	$f_b\circ \sigma ^a=\sigma^a\circ\sigma^{-b-a}\circ f\circ \sigma^{a+b}=\sigma^a\circ f_{a+b}$. 
	Further, $f_n^2=\operatorname{Id}$ and 
	\begin{align}\label{eq: commuting}f_n \circ f_m=f_m\circ f_n\qquad \text{for all } n,m\in \mathbb{Z}.
	\end{align}

  We denote by $\ZF$ the collection of all finite subsets of $\Z$;
  elements of $\ZF$ are denoted by bold face lower case letters.
  Given $\mathbf{b}=\{b_1,\ldots ,b_k\}\in \ZF$, we write
  $f_{\mathbf b}=f_{b_k}\circ \dots \circ f_{b_1}$.
  Note that due to \eqref{eq: commuting}, $f_{\bb}$ is well-defined.
 For the convenience of the reader, we recall the following basic fact about $G$. For more information regarding the Lamplighter group, see for instance \cite[Section~4]{BonanomeDeanDean2018}.

\begin{lem}
	The mapping $\mathbb{Z}\times [\Z]^{<\omega} \to G$ defined by $(a,\bb)\mapsto \sigma^a\circ f_{\bb}$ is bijective. 
\end{lem}
\begin{proof}
	From $f^2=\operatorname{Id}$ we observe that every element $g\in G$ is of the form 
	\[g=\sigma^{a_k}\circ f \circ \sigma^{a_{k-1}}\circ f \circ \dots \circ f\circ \sigma^{a_0},\] 
where $a_0,\dots, a_k \in \mathbb{Z}$. We write 
$b_{i}=\sum_{j=0}^{i-1} a_j$ for $i=1,\dots, k+1$. 
With $a=b_{k+1}$ and $\bb=\{b_i\colon \, i=1, \dots, k\}$, a straightforward computation shows 
$g=\sigma^{a}\circ f_{\bb}$. This yields surjectivity of the mapping under consideration. 

	To show injectivity, consider distinct $(a,\bb),(a',\bb')\in \mathbb{Z}\times [\Z]^{<\omega}$. Write $g=\sigma^a\circ f_{\bb}$ and $g'=\sigma^{a'} \circ f_{\bb'}$. 
	If $a\neq a'$, clearly $g\neq g'$. 
	If $a=a'$, then $\bb\neq \bb'$. From $\bb=\{b\in \mathbb{Z}\colon g(\hat{b})\in \check{X}\}$ and a similar statement about $\bb'$ and $g'$, we obtain $g\neq g'$ from $\bb\neq \bb'$. 
\end{proof}

\subsection{Left F\o lner sequences}

We next discuss certain left F\o lner sequences $(F_n)$ in $G$, where for each $n\in \N$, we will obtain $F_n$ as a product of elements of a 
F\o lner sequence in $\Z$ and subsets of
\[
	 I_n= \mc P([-2^{n},2^{n}]),
\]
that is, subsets of the power set of $[-2^{n},2^{n}]$.
Here and in the following, unless stated otherwise, given $m,n\in \Z_{\geq 0}$, we denote by $[-m,n]$ the respective interval in $\Z$, that is,
$[-m,n]=\{-m,-m+1,\ldots,n\}$.
We will make use of the following auxiliary statement whose proof we include for the convenience of the reader.

\begin{lem}\label{lem: construction of Vn}
 Let $(r_\ell)_{\ell\in\Z}$ be a real-valued sequence with $0\leq r_{\ell}\leq 1$.
 Then for each $n\in \N$, there is a
 collection $V_n\ssq \mc P([-n, n])$ such that
\begin{align}\label{eq: right fraction in Vn}
 \left|\frac{|\{\vv \in V_n\:\ell\in \vv\}|}{|V_n|}-r_\ell\right|\leq 1/n
 \qquad (\ell=-n,\ldots,n).
\end{align}
\end{lem}
\begin{proof}

We construct $(V_n)_{n\in \N}$ recursively
in a way which ensures $|V_{n+1}|>|V_n|\geq n$
and $|r_\ell^{(n)}-r_\ell|\leq 1/|V_n |$ for $\ell=-n,\ldots,n$,
where $r_\ell^{(n)}={|\{\vv \in V_n\:\ell\in \vv\}|}/{|V_n|}$.

The case $n=1$ is trivial.
Assuming we have already constructed $V_n$, we define $\ww=\{\ell\in[-n,n]\colon r_{\ell}^{(n)}<r_{\ell}\}$.

We now enlarge some elements of $V_n\cup \{\ww\}$ by adding $-(n+1)$.
Specifically, we pick $U\subseteq V_n$ such that $\ww\notin U$ and
$r_{-(n+1)}-1/|V_n|\leq |U|/|V_n|\leq r_{-(n+1)}$.
We set
\[
    V_n'=\{\uu\cup\{-n-1\}\colon \uu\in U\cup\{\ww\}\}\cup V_n\setminus U. 
\]
Similarly, we enlarge some sets of $V_n'$ by adding $n+1$ to obtain $V_{n+1}\in\mc P([-n-1, n+1])$ such that
$|V_{n+1}|>|V_n|$ and $|r_\ell^{(n+1)}-r_\ell|\leq 1/|V_{n+1}|$
holds for $\ell=-(n+1),\ldots,n+1$.
The statement follows.
\end{proof}

Given a real-valued sequence $r=(r_\ell)_{\ell\in\Z}$ with $0\leq r_{\ell}\leq 1$,
we choose $V_n$ as in Lemma~\ref{lem: construction of Vn} and set 
$B_n=B_n^r=\big\{\bb \in I_n\colon \bb\cap[-n,n]\in V_n\big\}$ as well as
\begin{align}\label{eq: defn Folner sequence in G}
	F_n=F_n^{r}=\{\sigma^a\circ f_{\mathbf b} \in G\:  a \in  [-2^n,2^n],\, \mathbf b \in B_n \}.
\end{align}
Note that due to \eqref{eq: right fraction in Vn}, for $\ell=-n,\ldots,n$, we have 
\begin{align}\label{eq: approximate ratio}
 \left|\frac{|\{\mathbf b \in B_n\:  \ell\in \mathbf b \}|}{|B_n|} - r_{\ell}\right|\leq 1/n.
\end{align}

\begin{lem}\label{lem: Fn is Folner}
	We have that $(F_n)_{n\in \mathbb{N}}$ with $F_n$ as in \eqref{eq: defn Folner sequence in G} is left F\o lner in $G$. 
\end{lem}

\begin{proof}
    Consider $g=\sigma^c\circ f_{\dd}\in G$ so that 
    $g\circ \sigma^a\circ f_{\bb}=\sigma^{c+a}\circ f_{\dd+a}\circ f_{\bb}$,
    where we write $\dd+a=\{d+a\colon d\in\dd\}$.
	Further, observe that 
    \begin{align*}
	    (g\circ F_n)\setminus F_n
	    &\ssq g\circ \bigcup_{t\in\dd}\{\sigma^a\circ f_{\mathbf b}\in F_n\: t+a\notin [-{2^{n}},{2^{n}}]\setminus [-n,n]\}\\
	    &\phantom{\ssq } \cup  g\circ 
	    \{\sigma^a\circ f_{\mathbf b}\in F_n\: c+a\notin [-2^n,2^n]\}\\
	    &\ssq g\circ \bigcup_{k\in\dd\cup\{c\}}\{\sigma^a\circ f_{\mathbf b}\in F_n\: a\notin ([-{2^{n}},{2^{n}}]\setminus [-n,n])-k\}.
    \end{align*}
	Therefore,
    \begin{align*}
    \frac{|(g\circ F_n)\setminus F_n|}{|F_n|}	
        &\leq \sum_{k\in \dd\cup\{c\}} \frac{|B_n|\cdot |[-2^n,2^n]\setminus ([-{2^{n}},{2^{n}}]\setminus [-n,n])-k|}{|B_n|\cdot|[-2^n,2^n]|}\\
        &= \sum_{k\in \dd\cup\{c\}} \frac{|[-2^n,2^n]\setminus [-{2^{n}},{2^{n}}]-k|}{|[-2^n,2^n]|}+
        \frac{|[-2^n,2^n]\cap [-n,n]-k|}{|[-2^n,2^n]|}
    \end{align*}
    which tends to $0$ as $n\to \infty$ since $[-2^n,2^n]$ defines a F\o lner sequence in $\Z$.
\end{proof}
\begin{rem}\label{rem: sequence of sequence of sequence}
    For furture reference, observe that a straightforward adaption of the above proof shows that for any sequence $\big(r^{(n)}\big)_{n\in\N}$ of sequences $r^{(n)}$ with values in the real interval $[0,1]$, we have that $F_n=F_n^{r^{(n)}}$ defines a (left) F\o lner sequence $(F_n)_{n\in\N}$. 
\end{rem}

\begin{rem}
 Some authors prefer to understand F\o lner sequences to be monotone and exhausting.
  Clearly, one can always obtain a monotone and exhausting F\o lner sequence $(F_n')$ from $(F_n^r)$ with the same asymptotics.
 Indeed, given a strictly increasing and sufficiently sparse sequence $(n_k)$ in $\N$, we may just set $F'_k=F^r_{n_{k+1}}\cup \{\sigma^a \circ f_{\mathbf b}\:a \in [-2^{n_k},2^{n_k}],\, {\mathbf b}\in I_{n_k}\}$.
\end{rem}

\begin{lem}\label{lem:non-ergodic-measures-I}
    Let $r=(r_\ell)_{\ell\in\Z}$ be a real-valued sequence with $0\leq r_{\ell}\leq 1$. Consider a (left) F\o lner sequence  $(F_n)$ given by \eqref{eq: defn Folner sequence in G}. 
    Then for $x=(b,i)\in X$, we have the following convergence (with respect to the weak*-topology) as $n\to \infty$ 
\begin{align}\label{eq: convergence of empirical measures in the example}
    \frac{1}{|F_n|} \cdot\sum_{g\in F_n} \delta_{gx}\rightarrow
\begin{cases}
(1-r_{b})\cdot \delta_{\hat \infty}+ r_{b}\cdot\delta_{\check \infty} & \text{ if } x\in \hat{X},\\
r_{b}\cdot \delta_{\hat \infty}+ (1-r_{b})\cdot \delta_{\check \infty}
& \text{ if } x\in \check{X}.
\end{cases}
\end{align}
\end{lem}
\begin{proof}
    W.l.o.g.\ we assume $x\in \hat{X}$, i.e.\ that $x=\hat{b}$. 
    With $A_n=[-2^n,2^n]$, we have
	\begin{align*}
    \frac{1}{|F_n|} \sum_{g\in F_n} \delta_{gx}
    &= \frac{1}{|A_n|\cdot |B_n|}\sum_{\mathbf b\in B_n}\, \sum_{a\in A_n}\delta_{\sigma^a\circ f_{\mathbf b}x}\\
    &= \frac{1}{|A_n|\cdot |B_n|}\sum_{\substack{a\in A_n\\ \mathbf b\in B_n,\, b\notin \mathbf b}} \delta_{\sigma^a\hat{b}}+
	\frac{1}{|A_n|\cdot |B_n|}\sum_{\substack{a\in A_n\\ \mathbf b \in B_n,\, b\in \mathbf b}} \delta_{\sigma^a\check{b}},
    \end{align*}
	for each $n\in \N$. 
	Observe that $(\hat X, \langle\sigma\rangle)$ and $(\check X, \langle\sigma\rangle)$, respectively, are uniquely ergodic so that
	$\hat b$ ($\check b$) is $\delta_{\hat{\infty}}$-generic ($\delta_{\check{\infty}}$-generic) along $(A_n)_{n\in \N}$. 
	Hence, $\frac{1}{|A_n|}\sum_{a\in A_n} \delta_{\sigma^a\hat b}\to\delta_{\hat{\infty}}$ and $\frac{1}{|A_n|}\sum_{a\in A_n} \delta_{\sigma^a\check b}\to\delta_{\check{\infty}}$. 
    Now, \eqref{eq: convergence of empirical measures in the example} follows from
	$|\{\mathbf b \in B_n\: b  \in \mathbf b\}|/|B_n|\to r_{b}$, see \eqref{eq: approximate ratio}.
\end{proof}
\begin{rem}
   Note that by considering F\o lner sequences $(F_n)_{n\in\N}$ as in Remark~\ref{rem: sequence of sequence of sequence},
    we can enforce essentially any kind of convergence or divergence of the measures $\frac{1}{|F_n|} \sum_{g\in F_n} \delta_{gx}$.
   This shows the first half of Theorem~\ref{thm: example}.
However, there is one inherent and unavoidable symmetry.
 As the computation in the above proof shows, along \emph{each} F\o lner sequence in $G$, $\frac{1}{|F_n|} \sum_{g\in F_n} \delta_{g(b,1)}$ accumulates at $\lambda \delta_{\hat \infty}+(1-\lambda) \delta_{\check \infty}$ if and only if
 $\frac{1}{|F_n|} \sum_{g\in F_n} \delta_{g(b,0)}$ accumulates at $(1-\lambda) \delta_{\hat \infty}+\lambda \delta_{\check \infty}$.
\end{rem}

\subsection{Right F\o lner sequences}

	Observe that unless $r_\ell=1/2$ for all $\ell\in\Z$, $(F_n^r)$ as defined in \eqref{eq: defn Folner sequence in G} is not right F\o lner.     
    Indeed, given $b\in \Z$, we have
	\[\liminf_{n\to \infty}|(F_n^r \circ f_b )\triangle F_n^r|/|F_n^r| \geq |2r_b-1|.\] 

    Note that, a priori, for a right F\o lner sequence $(F_n)$, the average $1/|F_n|\cdot\sum_{g\in F_n} \delta_{gx}$ does not need to converge to an invariant measure. Yet, in our example, we obtain the following. 

\begin{lem}\label{lem:non-ergodic-measures-II}
    Assume $(F_n)$ is a right F\o lner sequence. 
    Then
	\begin{itemize}
	    \item[(i)] For each $b\in\Z$,
            \begin{align*}
            \lim_{n\to\infty} \frac{|\{\sigma^a\circ f_{\mathbf b}\in F_n\: b\in \mathbf b \}|}{|F_n|}=
            \lim_{n\to\infty} \frac{|\{\sigma^a\circ f_{\mathbf b}\in F_n\: b\notin \mathbf b \}|}{|F_n|}=\frac{1}{2}.
            \end{align*}
        \item[(ii)] 
        For each $x\in X\setminus \{\hat \infty,\check \infty\}$, we have $1/|F_n|\cdot\sum_{g\in F_n} \delta_{gx} \to 1/2\cdot \delta_{\hat \infty}+ 1/2\cdot \delta_{\check \infty}$ .
	\end{itemize}
\end{lem}
\begin{proof}
Given any sequence $(F_n)$ of subsets in $G$, for each $b \in \Z$ and $n\in \N$, we have
	\begin{align*}
	 |F_n\circ f_b\, \setminus\, F_n|&\geq \big| |\{\sigma^a\circ f_{\mathbf b}\in F_n\: b\in \mathbf b \}|-|\{\sigma^a\circ f_{\mathbf b}\in F_n\: b\notin \mathbf b \}|\big |\\
	 &=
	 \big||F_n|- 2|\{\sigma^a\circ f_{\mathbf b}\in F_n\: b\in \mathbf b\}|\big|.
	\end{align*}
	Hence, if $(F_n)$ is right F\o lner, i.e. 
	$|F_n\circ f_b\, \setminus\, F_n|/|F_n|\to 0$,
	we obtain (i).
	
    Towards (ii), assume w.l.o.g.\ that $x\in \hat{X}$, i.e.\ that $x=\hat{b}$ for some $b\in \mathbb{Z}$. We have 
   \begin{align}\label{eq: folner average along right folner sequence}
     \frac{1}{|F_n|}\cdot\sum_{g\in F_n} \delta_{gx} = \frac{1}{|F_n|}\sum_{\substack{\sigma^a\circ f_{\mathbf b} \in F_n\\ b\notin \mathbf b}} \delta_{\sigma^a\hat{b}}+
	\frac{1}{|F_n|}\sum_{\substack{\sigma^a\circ f_{\mathbf b} \in F_n\\ b\in \mathbf b}} \delta_{\sigma^a\check{b}},
   \end{align}
	for each $n\in \N$.
    Observe that since $(F_n)$ is (right) F\o lner, we have that for each $\eps>0$ and $k_0\in \N$, there is $n_0$ such that for all $n\geq n_0$,
    \begin{align}\label{eq: Folner doesnt localise}
     |\{\sigma^a\circ f_{\mathbf b}\in F_n\: |a|<k_0\}|/|F_n|<\eps.
    \end{align}
    
   Now, given some continuous function $h$ on $X$ and $\eps>0$ (we may assume $\eps<1/4$), choose
   $k_0\in \N$ such that $|h(\hat \infty)-h(\widehat{b-a})|<\eps$ and $|h(\check \infty)-h(\widecheck{b-a})|<\eps$ for all $a\in \Z$ with $|a|\geq k_0$ and let $n_0$ be such that \eqref{eq: Folner doesnt localise} holds for all $n\geq n_0$.
   Due to (i), we may also assume without loss of generality that 
   $\big |1/2-|\{\sigma^a\circ f_{\mathbf b}\in F_n\: b\in \mathbf b \}|/|F_n|\big |<\eps $ and 
   $\big |1/2-|\{\sigma^a\circ f_{\mathbf b}\in F_n\: b\notin \mathbf b \}|/|F_n|\big |<\eps $ for all $n\geq n_0$.
   With \eqref{eq: folner average along right folner sequence}, we obtain for all $n\geq n_0$,
   \begin{align*}
    &|1/2\cdot \delta_{\hat \infty}(h) +1/2 \cdot \delta_{\check \infty}(h)-1/|F_n|\cdot\sum_{g\in F_n} \delta_{gx}(h)|\\
    &=\bigg|1/2\cdot  h(\hat \infty)+1/2\cdot h(\check \infty)-
    \frac{1}{|F_n|}\sum_{\substack{\sigma^a\circ f_{\mathbf b} \in F_n\\ b\notin \mathbf b}} h(\widehat{b-a})-
	\frac{1}{|F_n|}\sum_{\substack{\sigma^a\circ f_{\mathbf b} \in F_n\\ b\in \mathbf b}} h(\widecheck{b-a})\bigg|\\
	&\leq \bigg|1/2\cdot  h(\hat \infty)-
    \frac{1}{|F_n|}\sum_{\substack{\sigma^a\circ f_{\mathbf b} \in F_n\\ b\notin \mathbf b}} h(\widehat{b-a})\bigg|
    +\bigg|1/2\cdot h(\check \infty)-
	\frac{1}{|F_n|}\sum_{\substack{\sigma^a\circ f_{\mathbf b} \in F_n\\ b\in \mathbf b}} h(\widecheck{b-a})\bigg|\\
	&\leq \bigg|1/2\cdot  h(\hat \infty)-
	\frac{1}{|F_n|}\sum_{\substack{\sigma^a\circ f_{\mathbf b} \in F_n\\ b\notin \mathbf b}}h(\hat \infty)\bigg|
	+
	\bigg|\frac{1}{|F_n|}\sum_{\substack{\sigma^a\circ f_{\mathbf b} \in F_n\\ b\notin \mathbf b}}h(\hat \infty)- h(\widehat{b-a})\bigg |
    +\bigg|\ldots \bigg|\\
    &\leq\eps \cdot( h(\hat \infty)+h(\check \infty)+2).
   \end{align*}
As $\eps$ and $h$ were arbitrary, this finishes the proof.	
\end{proof}

\begin{rem}
   For the example under consideration, Lemma~\ref{lem:non-ergodic-measures-II} rules out continuity of the map $x\mapsto \mu_x$ if $\mu_x$ is obtained through averaging along a right F\o lner sequence. 
   This shows the second half of Theorem~\ref{thm: example}. 
\end{rem}

\appendix
\section{Ergodic theory with thin F\o lner sequences}\label{appendix: thin folner}
	The goal of this appendix is to convince the reader that some relevant parts of the basic machinery from ergodic theory are also available when working with thin F\o lner sequences.
	The discussed statements are well-known for $\sigma$-compact locally compact
	topological groups and (thick) F\o lner sequences, see for example \cite{EinsiedlerWard2011}.
	While most parts of their proofs immediately carry over to the case of thin F\o lner sequences, we provide some discussion of the slight deviations from the standard situation (and whenever there is no such discussion, the proofs are literally the same in both cases).

    In all of the following, 
	we adopt the notation and the standing assumptions from the main body of this article.
\begin{lem}\label{lem:prep_MET}
 Given a system $(X,G)$, for $f\in C(X)$ and $g\in G$,
	we have $\|S_n(g.f-f)\|_\infty \to 0$. 
\end{lem}
\begin{proof}
	Let $\eps>0$. 
	By continuity, 
	there exists an open neighbourhood $V$ of $e_G$ such that for all $g'\in V$, we have $\|g'. f-f\|_\infty <\eps/2$.	
By definition of the $V$-matching numbers, for $n\in \mathbb{N}$, there exists $E_n\subseteq F_n$ with $|E_n|=\mathfrak{m}_V(F_n ,g F_n)$ and a bijection $\phi_n\colon F_n\to gF_n$ such that for all $g'\in E_n$, we have $\phi_n(g'){g'}^{-1}\in V$ and hence,
	\[
	\|\phi_n(g').f-g'.f\|_\infty
	=\|{g'}^{-1}.(\phi_n(g').f)-f\|_\infty
	=\|\phi_n(g'){g'}^{-1}.f-f\|_\infty
	\leq \eps/2. 
	\]		
	
	Further, for large enough $n\in \mathbb{N}$, we have $\mathfrak{m}_V(gF_n ,F_n)/|F_n|\geq 1-\eps/(4\|f\|_\infty)$ and hence,
	\[
	\frac{|F_n\setminus E_n|}{|F_n|}=1-\frac{|E_n|}{|F_n|}<\frac{\eps}{4\|f\|_\infty}.
	\]	
	For such $n$, we compute
\begin{align*}
	\|S_n(g.f-f)\|_\infty
	&= \frac{1}{|F_n|} \Big\|\sum_{g'\in F_n} g'.(g.f)-g'.f\Big\|_\infty
	= \frac{1}{|F_n|} \Big\|\sum_{g'\in F_n} gg'.f-g'.f\Big\|_\infty\\
	&\leq \frac{1}{|F_n|} \sum_{g'\in E_n}  \|\phi_n(g').f-g'.f\|_\infty
	+\frac{1}{|F_n|} \sum_{g'\in F_n\setminus E_n} 2\|f\|_\infty\\
	&\leq \frac{|E_n|}{|F_n|}\frac{\eps}{2} +\frac{|F_n\setminus E_n|}{|F_n|}2\|f\|_\infty\leq \eps. 
\end{align*}
As $\eps>0$ was arbitrary, the statement follows. 
\end{proof}

\begin{cor}[Krylov-Bogolyubov]\label{cor: Krylov-Bogolyubov}
	Given a dynamical system $(X,G)$ and $x\in X$, then every limit point of $(S_n^*\delta_x)$ is contained in $\mathcal{M}(X,G)$. 
\end{cor}
\begin{proof}
	Let $\mu$ be a limit point and $f\in C(X)$. 
	We have
	\begin{align*}
		|{g}^\ast\mu(f)-\mu(f)|
		&\leq |{g}^\ast\mu(f)-{g}^\ast S_n^*\delta_x(f)| + |{g}^\ast S_n^*\delta_x(f)-S_n^*\delta_x(f)| + |S_n^*\delta_x(f)-\mu(f)|\\
		&\leq |\mu(g.f)-S_n^*\delta_x(g.f)| + \|S_n(g.f-f)\|_\infty + |S_n^*\delta_x(f)-\mu(f)|
		\to 0. 
	\end{align*}
	Thus, ${g}^*\mu=\mu$ as claimed. 
\end{proof}
	Based on the above, one obtains
\begin{thm}\label{thm: uniform ergodic theorem}
Let $(X,G)$ be a dynamical system.
    The following are equivalent. 
	\begin{itemize}
	\item[(i)] The system $(X,G)$ is uniquely ergodic. 
	\item[(ii)] For every $f\in C(X)$, there is a constant $c$ such that for some F\o lner sequence $(F_n)_{n\in\N}$ and every $x\in X$, we have $(S_nf)(x)\to c$. 
	\end{itemize}
	
    Further, if one of the above conditions holds, then the convergence in (ii)
    is uniform in $x\in X$, independent of the specific F\o lner sequence $(F_n)_{n\in\N}$, and we have $c=\mu(f)$.
\end{thm}

\begin{prop}[{cf. \cite[Proposition~2.3]{FuhrmannGroegerLenz2022}}] \label{prob:right Foelner implies unique ergodicity}
    Let $(X,G)$ be a dynamical system.
    Suppose for each $f \in C(X)$ there is a right Følner sequence $(F_n)_{n\in\N}$ and a constant $c$ with
    \[
        \lim_{n\to\infty}\frac{1}{\theta_r(F_n)}\int\limits_{F_n}f(gx)\,d\theta_r(g)=c,
    \]
    for all $x\in X$.
    Then $(X,G)$ has a unique $G$-invariant measure $\mu$ and $\mu(f)=c$.
\end{prop}

	In the following, we denote by $C(\mu)$ the image of $C(X)$ under the canonical mapping into $L^2(\mu)$.\footnote{Note that this mapping is not necessarily injective and hence, not necessarily an embedding.}
	For $p=1,2$, denote by $I^p(\mu)$ the set of all $f\in L^p(\mu)$ with $g.f =f$ for all $g \in G$. 
	Denote by $P_{\mu}$ the projection onto the subspace $I^2(\mu)$ in $L^2(\mu)$. 	
	
	For the convenience of the reader, we briefly discuss the part of the proof of the next statement where the case of thin F\o lner sequences slightly deviates from the standard case.
	
\begin{thm}[Mean Ergodic Theorem]
	Let $(X,G)$ be a dynamical system and $\mu \in \mc M(X,G)$.
	For $f\in L^2(\mu)$, we have $S_n f \to P_{\mu}f$ in $L^2(\mu)$. 
\end{thm}
\begin{proof}
	Define $A=\{g.f-f\: f\in C(\mu), g\in G\}$. 
	Since $C(\mu)$ is dense in $L^2(\mu)$, one can show that $L^2(\mu)=I^2(\mu)\oplus \overline{A}$, where $\overline{A}$ denotes the closure of $A$. 
	Now, recall that $\|f\|_{2}\leq \|f\|_\infty$ for all $f\in C(\mu)$. 
	Thus, with Lemma~\ref{lem:prep_MET}, one can show that for $f\in \overline{A}$, we have $\|S_n f\|_2\to 0$.
	Since $S_n f=f$ for all $f\in I^2(\mu)$, we have $S_n f\to P_{\mu} f$
	for each $f\in L^2(\mu)$. 	
\end{proof}

\begin{cor}[$L^1$-Mean Ergodic Theorem]
Let $(X,G)$ be a dynamical system and $\mu \in \mc M(X,G)$.
	For $f\in L^1(\mu)$, we have that $S_n f$ converges in $L^1(\mu)$ to an element in $I^1(\mu)$. 
	In particular, if $\mu$ is ergodic, then $S_n f\to \mu(f)$.
\end{cor} 
	
\begin{cor}[{\cite[Theorem 2.4]{FuhrmannGroegerLenz2022}}]\label{cor: almost all points are generic}
    Let $(X,G)$ be a dynamical system with an ergodic measure $\mu$.
    Then every Følner sequence $\Fol$ allows for a subsequence $\Fol'$ 
    such that $\mu$-almost every point is $\mu$-generic along $\mc F'$.
\end{cor}

\end{document}